\documentclass[preprint,12pt,3p]{elsarticle}

\usepackage{amssymb}
\usepackage{graphicx}
\usepackage{amsmath}

\usepackage{caption}
\journal{Journal of Latex templates}

\newtheorem{theorem}{Theorem}
\newtheorem{prop}{Proposition}
\newtheorem{conjecture}{Conjecture}
\newtheorem{lem}{Lemma}
\newtheorem{corollary}{Corollary}

\newcommand{\newcaption}[2]{\begin{center} \captionof{figure}{#1}\label{#2}\end{center}}

\newenvironment{proof}{ {\bf Proof:}} {$\Box$}

\newenvironment{case}{ {\it Case }}

\newenvironment{subcase}{ {\it Subcase }}

\begin{document}

\begin{frontmatter}

\title{On Sylvester Colorings of Cubic Graphs}
\tnotetext[label0]{Part of the results of this paper were presented in CID 2013}

\author{Anush Hakobyan}
\ead{ashunik94@gmail.com}

\author{Vahan Mkrtchyan\corref{cor1}}
\ead{vahanmkrtchyan2002@ysu.am}
\cortext[cor1]{Corresponding author}
%\author{}
\address{Department of Informatics and Applied Mathematics, Yerevan State University, Yerevan, 0025, Armenia}

%\address[label2]{Address Two\fnref{label4}}

%\cortext[cor1]{I am corresponding author}
%\fntext[label3]{I also want to inform about\ldots}
%\fntext[label4]{Small city}

%\ead{ashunik94@gmail.com}
%\ead[url]{author-one-homepage.com}

%\author{Vahan Mkrtchyan}
%\cortext[cor2]{Part of the results of this paper were presented in CID 2013}

%\address[label5]{Department of Informatics and Applied Mathematics, Yerevan State University, Yerevan, Armenia \fnref{label1}}

%\author[label1,label5]{Author Three}
%\ead{author.three@mail.com}

\begin{abstract}
\medskip
If $G$ and $H$ are two cubic graphs, then an $H$-coloring of $G$ is a proper edge-coloring $f$ with edges of $H$, such that for each vertex $x$ of $G$, there is a vertex $y$ of $H$ with
$f(\partial_G(x))=\partial_H(y)$. If $G$ admits an $H$-coloring, then we will write $H\prec G$.
The Petersen coloring conjecture of Jaeger states that for any bridgeless cubic graph $G$, one
has: $P\prec G$. The second author has recently introduced the Sylvester coloring conjecture, which
states that for any cubic graph $G$ one has: $S\prec G$. Here $S$ is the Sylvester
graph on $10$ vertices. In this paper, we prove the analogue of Sylvester coloring conjecture for
cubic pseudo-graphs. Moreover, we show that if $G$ is any connected simple cubic graph $G$ with
$G\prec P$, then $G = P$. This implies that the Petersen graph does not admit an $S_{16}$-coloring, where $S_{16}$ is the smallest connected simple cubic graph without a perfect matching. $S_{16}$ has $16$ vertices. %We conjecture that there are infinitely many  connected cubic simple graphs which do not admit an
%$S_{16}$-coloring. 
Finally, we obtain $2$ results towards the Sylvester coloring conjecture. The first result states that any cubic graph $G$ has a coloring with edges of
Sylvester graph $S$ such that at least $\frac45$ of vertices of $G$ meet the conditions of Sylvester coloring conjecture. The second result states that any claw-free cubic graph graph admits an $S$-coloring. This results is an application of our result on cubic pseudo-graphs.
\end{abstract}

\begin{keyword}
Cubic graph; Petersen graph; Petersen coloring conjecture; Sylvester graph; Sylvester coloring
conjecture
\end{keyword}

\end{frontmatter}

% \linenumbers

\section{Introduction}
\label{int1}
Graphs considered in this paper are finite and undirected. They do not contain loops,
though they may contain parallel edges. We also consider pseudo-graphs, which may
contain both loops and parallel edges, and simple graphs, which contain neither loops nor parallel edges. As usual, a loop contributes to the degree of a vertex by $2$.

Within the frames of this paper, we assume that graphs, pseudo-graphs and simple graphs are considered up to isomorphisms. This implies that the equality $G=G'$ means that $G$ and $G'$ are isomorphic.

For a graph $G$, let $V(G)$ and $E(G)$ be the set of vertices and edges of $G$, respectively. Moreover, let $\partial_{G}(x)$ be the set of edges of $G$ that are incident to the vertex $x$ of $G$. A matching of $G$ is a set of edges of $G$ such that any two of them do not share a vertex. A matching of $G$ is perfect, if it contains $\frac{|V(G)|}{2}$ edges. For a positive integer $k$, a $k$-factor of $G$ is a spanning $k$-regular subgraph of $G$. Observe that the edge-set of a $1$-factor of $G$ is a perfect matching of $G$. Moreover, if $G$ is cubic and $F$ is a $1$-factor of $G$, then the set $E(G)\backslash E(F)$ is an edge-set of a $2$-factor of $G$. This $2$-factor is said to be complementary to $F$. Conversely, if $\bar{F}$ is a $2$-factor of a cubic graph $G$, then the set $E(G)\backslash E(\bar{F})$ is an edge-set of a $1$-factor of $G$ or is a perfect matching of $G$. This $1$-factor is said to be complementary to $\bar{F}$.

If $P$ is a path of a graph $G$, then the length of $P$ is the number of edges of $G$ lying on $P$. For a connected graph $G$ and its two vertices $u$ and $v$, the distance between $u$ and $v$ is the length of the shortest path connecting these vertices. The distance between edges $e$ and $f$ of $G$, denoted by $\rho_G(e,f)$, is the shortest distance among end-vertices of $e$ and $f$. Clearly, adjacent edges are at distance zero.

A subgraph $H$ of $G$ is even, if every vertex of $H$ has even degree in $H$. A block of $G$ is a maximal $2$-connected subgraph of $G$. An end-block is a block of $G$ containing at most one vertex that is a cut-vertex of $G$. If $G$ is a cubic graph containing cut-vertices, then any end-block $B$ of $G$ is adjacent to a unique bridge $e$. We will refer to $e$ as a bridge corresponding to $B$. Moreover, if $e=(u,v)$ and $u\in V(B)$, $v\notin V(B)$, then $v$ is called the root of $B$.

If $G$ is a cubic graph, and $K$ is a triangle in $G$, then one can obtain a cubic pseudo-graph by contracting $K$. We will denote this pseudo-graph by $G/K$. If $G/K$ is a graph, we will say that $K$ is contractible. Observe that if $K$ is not contractible, two vertices of $K$ are joined with two parallel edges, and the third vertex is incident to a bridge (see the end-blocks of the graph from Figure \ref{SylvGraph}). If $K$ is a contractible triangle, and $e$ is an edge of $K$, then let $f$ be the edge of $G$ that is incident to a vertex of $K$ and is not adjacent to $e$. $e$ and $f$ will be called opposite edges.

If $T$ is a set, $H$ is a subgraph of a graph $G$, and $f: E(G)\rightarrow T$, then a mapping $g: E(H)\rightarrow T$, such that $g(e)=f(e)$ for any $e\in E(H)$ is called the restriction of $f$ to $H$.

Let $G$ and $H$ be two cubic graphs, and let $f:E(G)\rightarrow E(H)$. Define:

$$ V(f)=\{x\in V(G):\exists y\in V(H) \text{ } f(\partial_{G}(x)) = \partial_{H}(y)\}.$$

An $H$-coloring of $G$ is a mapping $f:E(G)\rightarrow E(H)$, such that $V(f)=V(G)$. If $G$ admits an $H$-coloring, then we will write $H
\prec G$. It can be easily seen that if $H\prec G$ and $K\prec H$, then $K\prec G$. In other words, $\prec$ is a transitive relation defined on the set of cubic graphs.

If $H \prec G$ and $f$ is an $H$-coloring of $G$, then for any adjacent edges $e$, $e'$
of $G$, the edges $f(e)$, $f(e^{\prime})$ of $H$ are adjacent. Moreover, if the graph $H$
contains no triangle, then the converse is also true, that is, if a mapping $f : E(G) \to E(H)$
has a property that for any two adjacent edges $e$ and $e^{\prime}$ of $G$, the edges $f(e)$
and $f(e^{\prime})$ of $H$ are adjacent, then $f$ is an $H$-coloring of $G$.

\begin{center}
    \includegraphics [scale=0.5] {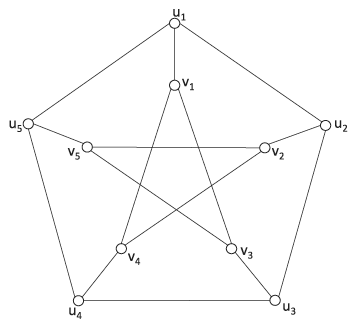}
    \newcaption{The Petersen Graph}{PetersenGraph}
\end{center}

Let $P$ be the well-known Petersen graph (Figure \ref{PetersenGraph}) and let $S$ be the graph from Figure \ref{SylvGraph}. $S$ is called the Sylvester graph \cite{Schrijver}. We would like to point out that usually the name “Sylvester graph” is used for a particular strongly regular
graph on $36$ vertices, and this graph should not be confused with $S$, which has $10$
vertices.

The Petersen coloring conjecture of Jaeger states:
\begin{conjecture} (Jaeger, 1988 \cite{Jaeger}) For each bridgeless cubic
graph $G$, one has $P \prec G$.
\end{conjecture}
%This means, that every bridgeless cubic graph has $P$-coloring.
The conjecture is difficult to prove, since it can be seen that it implies the
following two classical conjectures:
\begin{conjecture} (Berge-Fulkerson, 1972 \cite{Fulkerson,Seymour}) Any bridgeless
cubic graph $G$ contains six (not necessarily distinct) perfect matchings
$F_1, \ldots , F_6$ such that any edge of $G$ belongs to exactly two of them.
\end{conjecture}

\begin{center}
    \includegraphics [scale=0.4] {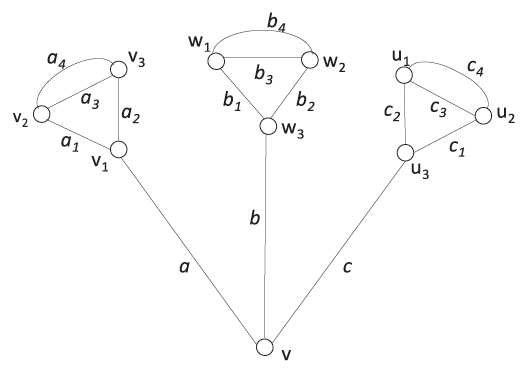}
    \newcaption{The Sylvester Graph}{SylvGraph}
\end{center}

\begin{conjecture}
((5, 2)-cycle-cover conjecture, \cite{Celmins1984,Preiss1981}) Any bridgeless
graph $G$ (not necessarily cubic) contains five even subgraphs such that any
edge of $G$ belongs to exactly
two of them.
\end{conjecture}

\begin{center}
    \includegraphics [scale=0.5] {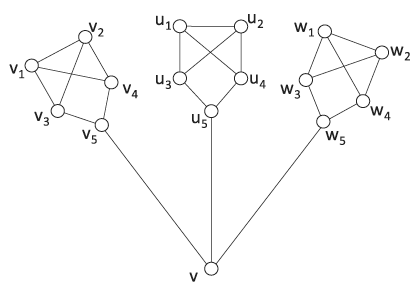}
    \newcaption{The Graph $S_{16}$}{SylvSimple}
\end{center}

Related with the Jaeger conjecture, the following conjecture has been introduced in \cite{PetersenRemark}:
\begin{conjecture}(V. V. Mkrtchyan, 2012 \cite{PetersenRemark})
\label{mainConjecture}
For each cubic graph $G$, one has $S \prec G$. 
\end{conjecture}
In direct analogy with the Jaeger conjecture, we call Conjecture
\ref{mainConjecture} the Sylvester coloring conjecture.

In this paper, we consider the analogues of this conjecture for simple cubic graphs and cubic pseudo-graphs. Let $S_{16}$ be the simple graph from Figure
\ref{SylvSimple}, and let $S_4$ be the pseudo-graph from Figure \ref{SylvPseudo}.

%\begin{conjecture} 
%F%or each simple cubic graph $G$, one has $S_{16} \prec G$. 
%\end{conjecture}

\begin{center}
    \includegraphics [scale=0.5] {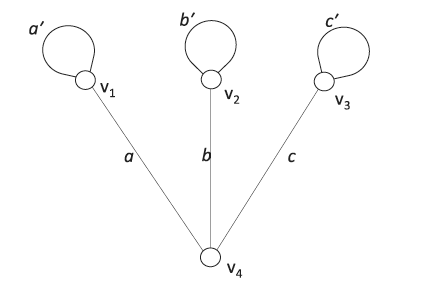}
    \newcaption{The Pseudo-Graph $S_4$}{SylvPseudo}
\end{center}

In this paper, we show that not all simple cubic graphs admit an $S_{16}$-coloring. On the positive side, we prove that all cubic pseudo-graphs have an $S_4$-coloring. We complete the paper by proving $2$ results towards Conjecture \ref{mainConjecture}. The first one states that for any cubic graph $G$ there is a mapping $f: E(G)\rightarrow E(S)$, such that \mbox{$|V(f)|\geq \frac{4}{5}\cdot |V(G)|$}. The second one states that any claw-free cubic graph admits an $S$-coloring. The latter result is derived as a consequence of the $S_4$-colorability of cubic pseudo-graphs.

Terms and concepts that we do not define in the paper can be found in \cite{Harary,West}.

\section{Some Auxiliary Statements}
\label{aux2}

In this section, we present some auxiliary statements that will be used in Section \ref{main3}. 

\begin{theorem}\label{PetersenTheorem}(Petersen, 1891 \cite{Lovasz})Let $G$ be a cubic graph containing at most two bridges. Then $G$ has a $1$-factor.
\end{theorem}

\begin{lem}\label{EndBlocks} Let $G$ be a bridgeless graph with $d(v)\in \{2,3\}$ for any $v\in V(G)$. Assume that all vertices of $G$ are of degree $3$ except one. Then $G$ has a $2$-factor.
\end{lem}

\begin{proof} Take two copies $G_1$ and $G_2$ of $G$, and consider a graph $H$ obtained from them by joining degree $2$ vertices by an edge $e$. Observe that $H$ is a cubic graph containing only one bridge, which is the edge $e$. By Theorem \ref{PetersenTheorem}, $H$ contains a $1$-factor $F$. Since $e$ is a bridge of $H$, we have $e\in F$. Consider the complementary $2$-factor $\bar{F}$ of $F$. Clearly, the edges of the set $E(\bar{F})\cap E(G_1)$ form a $2$-factor of $G_1$, which completes the proof of the lemma.
\end{proof}

%% A proposition %%
\begin{prop}\label{property}
Let $G$ be a simple cubic graph that has no a perfect matching and ${\left\vert V(G)
\right\vert \leq 16}$. Then ${G = S_{16}}$.
\end{prop}

%% A lemma %%
\begin{lem} \label{Lemma 1}
Suppose that $G$ and $H$ are cubic graphs with $H \prec G$, and let $f$ be an $H$-coloring of
$G$. Then:
\begin{enumerate}[(a)]
\item If $M$ is any matching of $H$, then $f^{-1}(M)$ is a matching of $G$;

\item $\chi^{\prime}(G) \leq \chi^{\prime}(H)$, where $\chi^{\prime}(G)$ is the chromatic index
of $G$;

\item If $M$ is a perfect matching of $H$, then $f^{-1}(M)$ is a perfect matching of $G$;

\item For every even subgraph $C$ of $H$, $f^{-1}(C)$ is an even subgraph of $G$;

\item For every bridge $e$ of $G$, the edge $f(e)$ is a bridge of $H$.

\end{enumerate}
\end{lem}

\begin{prop}\label{Non3edgeProp} Let $G$ be a connected non-$3$-edge-colorable simple cubic graph such that $|V(G)|\leq 10$. Then $G=P$ or $G=S'$ (Figure \ref{theorem3f3}).
\end{prop}

We will also need some results that were obtained in \cite{Fouquet,Steffen1,Steffen2}. Let $G$ be a graph of maximum degree at most $3$, and assume that $c$ is a proper coloring of some edges of $G$ with colors $1$, $2$ and $3$. The edges of $G$ that have not received colors in $c$ are called uncolored edges. Now, assume that $c$ is chosen so that the number of uncolored edges of $G$ is minimized. It is known that for such a choice of $c$, uncolored edges must form a matching. 

\begin{center}
    \includegraphics [scale=0.6] {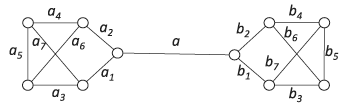}
    \newcaption{The graph $S'$ on $10$ vertices}{theorem3f3}
\end{center}

Take an arbitrary uncolored edge $e=(u,v)$. As $c$ is chosen so that the number of uncolored is smallest, we have that $e$ is incident to edges that have colors $1$, $2$, and $3$. Since $G$ is of maximum degree at most $3$, we have that there are colors $\alpha, \beta \in \{1, 2, 3\}$, such that no edge incident to $u$ and $v$ is colored with $\beta$ and $\alpha$, respectively. Consider a maximal $\alpha-\beta$-alternating path $P_e$ starting from $u$. As it shown in \cite{Fouquet,Steffen1,Steffen2}, this path must terminate in $v$, and hence it is of even length. This means that $P_e$ together with the edge $e$ forms an odd cycle $C_e$. $C_e$ is called the cycle corresponding to the uncolored edge $e$. It is known that

\begin{lem}\label{VertexDisjoint}(\cite{Fouquet,Steffen1,Steffen2})
If $G$ is a graph of maximum degree at most $3$, and $e$, $e'$ ($e\neq e'$) are $2$ uncolored edges of $G$, then $V(C_{e})\cap V(C_{e'})=\emptyset$.
\end{lem}

Finally, we will need some results on claw-free cubic graphs. Recall that a graph $G$ is claw-free, if it does not contain $4$ vertices, such that the subgraph of $G$ induced on these vertices is isomorphic to $K_{1,3}$. In \cite{ChudSeyClawFreeChar}, arbitrary claw-free graphs are characterized. In \cite{sang-il_oum:2011}, Oum has characterized simple, claw-free bridgeless cubic graphs. Following the approach of Oum, below we will characterize claw-free (not necessarily bridgeless) cubic graphs. 

We need some definitions. A $2$-cycle is a cycle of length $2$ ($2$ parallel edges). In a claw-free cubic graph $G$ any vertex belongs to $1$, $2$, or $3$ triangles or $1$ or $3$ $2$-cycles. If a vertex $v$ belongs to $3$ triangles of $G$, then the component of $G$ containing $v$ is isomorphic to $K_4$ (Figure \ref{K4K23}). An induced subgraph of $G$ that is isomorphic to $K_4-e$ is called a diamond \cite{sang-il_oum:2011}. It can be easily checked that in a claw-free cubic graph no $2$ diamonds intersect. If $v$ belongs to $3$ $2$-cycles of $G$, then the component of $G$ containing $v$ is isomorphic to $K_{2}^{3}$ (Figure \ref{K4K23}).

 \begin{center}
    \includegraphics [scale=0.5] {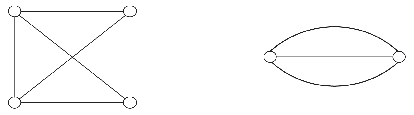}
    \newcaption{The graphs $K_4$ and $K_{2}^{3}$.}{K4K23}
\end{center}

A string of diamonds or $2$-cycles of $G$ is a maximal sequence $F_{1},...,F_{k}$ of diamonds or $2$-cycles, in which $F_{i}$ has a vertex adjacent to a vertex of $F_{i+1}$, $1\leq i \leq k-1$ (Figure \ref{stringk4k23}).  A string of diamonds or $2$-cycles has exactly $2$ vertices of degree $2$, which are called the head and the tail of the string. A string $J$ of a claw-free cubic graph $G$ is trivial, if $J$ is comprised of $1$ $2$-cycle, and $G$ contains a vertex that is adjacent to both the head and tail of $J$. Replacing an edge $e = (u,v)$ with a string of diamonds or $2$-cycles with the head $x$ and the tail $y$ is to remove $e$ and add edges $(u,x)$ and $(v,y)$.

 \begin{center}
    \includegraphics [scale=0.5] {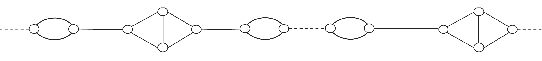}
    \newcaption{A string of diamonds or $2$-cycles.}{stringk4k23}
\end{center}

If $G$ is a connected claw-free cubic graph such that each vertex lies in a diamond or a $2$-cycle, then $G$ is called a ring of diamonds or $2$-cycles. It can be easily checked that each vertex of a ring of diamonds or $2$-cycles lies in exactly $2$ diamonds or $2$-cycles. As in \cite{sang-il_oum:2011}, we require that a ring of diamonds or $2$-cycles contains at least $2$ diamonds or $2$-cycles.

We are ready to present the characterization of claw-free cubic graphs.

\begin{prop}\label{prop:ClawfreeCharac} $G$ is a connected claw-free cubic graph, if and only if
\begin{enumerate}
    \item [(1)] $G$ is isomorphic to $K_4$ or $K_{2}^{3}$, or
    
    \item [(2)] $G$ is a ring of diamonds or $2$-cycles, or
    
    \item [(3)] there is a cubic pseudo-graph $H$, such that $G$ can be obtained from $H$ by replacing some edges of $H$ with strings of diamonds or $2$-cycles, and by replacing any vertex of $H$ with a triangle.
\end{enumerate}
\end{prop}

\begin{proof}
We omit the proof of the proposition, since it can be done in the same way as the proof of Proposition $1$ of \cite{sang-il_oum:2011}. 
\end{proof}

\section{The Main Results}
\label{main3}

In this section, we obtain the main results of the paper. Our first theorem shows that the statement analogous to Sylvester coloring conjecture holds for cubic pseudo-graphs. Actually, we will show that all cubic pseudo-graphs $G$ admit an $S_4$-coloring. We use the labels of edges of $S_4$ from Figure \ref{SylvPseudo}. %In the last statement, the term $S_4$-coloring should be understood as follows: a loop of $G$ is colored by a loop of $S_4$, and if a loop of $S_4$ is to be used in an ${S_4}$-coloring of $G$, then in the neighborhood of the corresponding vertex of $G$ it is used twice.

In the theorem, we need the concept of a list. A list is a collection of elements, where the order of elements is not important and elements may appear more than once. $2$ lists are equal, if they are comprised of the same elements, and the frequency of appearance of each element in the lists is the same. In contrast with sets, that we denote by $\{...\}$, lists will be denoted by $\langle ... \rangle$. According to our definitions, we have $\langle 1,1,2 \rangle=\langle 1,2,1 \rangle$ and $\langle 1,2,2 \rangle\neq \langle 2,1,1 \rangle$.

%% Result 1 %%

\begin{theorem}\label{thm:PseudoCubicS4}
Let ${G}$ be a cubic pseudo-graph, and let $\partial(S_4)$ be the following set of lists:\[\partial(S_4)=\{\langle a,b,c \rangle, \langle a,a',a' \rangle, \langle b,b',b' \rangle, \langle c,c',c' \rangle\}.\] Then, there is a mapping $f:E(G)\rightarrow E(S_4)$, such that
\begin{enumerate}
    \item [(a)] if a vertex $v$ of $G$ is incident to a loop $e'$ and a bridge $e$, then the list $\langle f(e),f(e'),f(e') \rangle$ is one of the $3$ lists of $\partial(S_4)\backslash \{\langle a,b,c \rangle\}$,
    
    \item [(b)] if a vertex $v$ is incident to $3$ edges $e$, $e'$ and $e''$, then the list $\langle f(e),f(e'),f(e'') \rangle$ is one of the $4$ lists of $\partial(S_4)$.
\end{enumerate}
\end{theorem}

\begin{proof} It is clear that we can prove the theorem only for connected cubic pseudo-graphs $G$. First of all, we prove the theorem when $G$ is a connected graph. We proceed by induction on the number of bridges of $G$. %We will use the labels of edges of $S_4$ given in Figure \ref{SylvPseudo}.

If there are at most $2$ bridges in $G$, then due to Theorem \ref{PetersenTheorem}, the graph
${G}$ has a $1$-factor. Color the edges of the $1$-factor by ${a}$, and the edges of the complementary
$2$-factor by ${a^\prime}$. It is not hard to see that the described coloring meets the condition (b) of the theorem.

Now assume that the statement is true for graphs with at most ${k}$ bridges,
and we prove it for those with ${k+1}$ bridges ($k+1\geq 3$). We will consider the following cases:

\medskip

\begin{case} 1: For any two end-blocks $B$ and $B'$ of $G$, the bridges $e$ and $e'$ corresponding to them, are adjacent.

\medskip

It is not hard to see that in this case, $G$ consists of three end-blocks $B_1$, $B_2$, $B_3$, such that the bridges $e_1$, $e_2$, $e_3$ corresponding to them are incident to the same cut-vertex $v$. We obtain an $S_4$-coloring of $G$ as follows: let $\bar{F}_1$, $\bar{F}_2$, $\bar{F}_3$ be $2$-factors in $B_1$, $B_2$, $B_3$, respectively (see Lemma \ref{EndBlocks}). Color the edges of $\bar{F}_1$ with $a'$ and the edges of $(E(B_1)\backslash \bar{F}_1)\cup \{e_1\}$ with $a$, the edges of $\bar{F}_2$ with $b'$ and the edges of $(E(B_2)\backslash \bar{F}_2)\cup \{e_2\}$ with $b$, the edges of $\bar{F}_3$ with $c'$ and the edges of $(E(B_3)\backslash \bar{F}_3)\cup \{e_3\}$ with $c$. It is not hard to see that the described coloring meets the condition (b) of the theorem.
\end{case}

\medskip

\begin{case} 2: There are two end-blocks $B$ and $B'$ of $G$, such that the bridges $e$ and $e'$ corresponding to them, are not adjacent. 

 Assume that $e=(u,u')$, $e'=(v,v')$ and $u'\in V(B)$, $v'\in V(B')$. Consider a cubic graph $H$ obtained from $G$ as follows:
 
 \begin{equation*}
 H=[G\backslash (V(B_1)\cup V(B_2))]\cup \{(u,v)\}.
 \end{equation*}If initially $G$ had an edge $(u, v)$, then in $H$
we will just get two parallel edges between $u$ and $v$.
 
Observe that $H$ is a cubic graph containing at most $k$ bridges. By induction hypothesis, $H$ admits an 
${S_4}$-coloring $g$ satisfying condition (b) of the theorem. Now we obtain an $S_4$-coloring for the graph $G$ using the coloring $g$ of $H$. For that purpose we consider $2$ cases.

\medskip

\begin{subcase} 2.1: $g((u,v))\in \{a, b, c\}$.

\medskip

Without loss of generality, we can assume that $g((u,v))=a$. Other cases can be come up in a similar way. By Lemma \ref{EndBlocks}, $B$ and $B'$ contain $2$-factors $\bar{F}$ and $\bar{F}'$, respectively. Consider the restriction of $g$ to $G$. We extend it to an $S_4$-coloring of $G$ as follows: color the edges of $\bar{F}\cup \bar{F}'$ with $a'$ and the edges of $(E(B)\backslash \bar{F})\cup (E(B')\backslash \bar{F}')\cup \{e, e'\}$ with $a$. It is not hard to see that the described coloring meets the condition (b) of the theorem.

\end{subcase}

\medskip

\begin{subcase} 2.2: $g((u,v))\in \{a', b', c'\}$.

\medskip

Without loss of generality, we can assume that $g((u,v))=a'$. Other cases can be come up in a similar way. It is not hard to see that the edges of $H$ colored with $a'$ (the edges of set $f^{-1}(a')$) form vertex disjoint cycles in $H$. Consider the cycle $C_{uv}$ of $G$ containing the edge $(u,v)$, and let $P_{uv}=C_{uv}-(u,v)$ (Figure \ref{theorem1f2}). By Lemma \ref{EndBlocks}, $B$ and $B'$ contain $2$-factors $\bar{F}$ and $\bar{F}'$, respectively. 

\begin{center}
\includegraphics [scale=0.6] {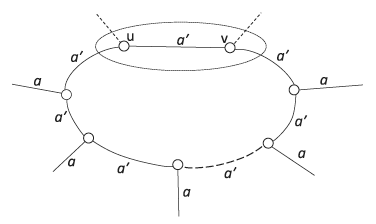}
\newcaption{The cycle $C_{uv}$ and the path $P_{uv}$ in the graph $H$}{theorem1f2}
\end{center}

Define edges $d$ and $d'$ of $S_4$ as follows: if $P_{uv}$ is of odd length, then $d=b$, $d'=b'$, and $d=c$, $d'=c'$, otherwise. Consider the restriction of $g$ to $G$. We extend it to an $S_4$-coloring of $G$ as follows: color the edges of $\bar{F}$ with $b'$ and the edges of $E(B)\cup \{e\}$ with $b$, re-color the edges of $P_{uv}$ by coloring them with colors $b$ and $c$ alternatively beginning from $c$, color the edges of $\bar{F}'$ with $d'$ and the edges of $E(B')\cup \{e'\}$ with $d$ (Figure \ref{theorem1f3}). It is not hard to see that the described coloring meets the condition (b) of the theorem.

\begin{center}
\includegraphics [scale=0.6] {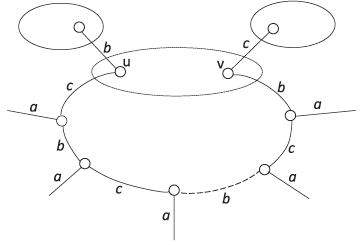}
\newcaption{The path $P_{uv}$ in the graph $G$}{theorem1f3}
\end{center}
\end{subcase}
\end{case}

 Finally, we consider the case when $G$ contains loops. Let $H$ be a graph obtained from $G$ by replacing all vertices of $G$ incident to loops, by triangles (Figure \ref{changeLoop}).

\begin{center}
    \includegraphics [scale=0.5] {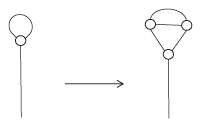}
    \newcaption{Modification of loops of $G$}{changeLoop}
\end{center}

Now ${H}$ has an ${S_4}$-coloring $f$ satisfying the condition (b) of the theorem. We claim that any bridge of $H$ has a color from $\{a, b, c\}$. 

For the sake of contradiction, assume that there is a bridge $e$ of $H$ such that $f(e)=a'$ (the cases $f(e)=b'$ and $f(e)=c'$ can be come up in a similar way). It is not hard to see that the edges of $H$ colored with $a'$ (the edges of set $f^{-1}(a')$) form vertex disjoint cycles in $H$. Consider the cycle $C$ containing $e$. Since $e$ is a bridge, $C$ intersects this cut in $1$ edge, which is a contradiction to the fact that cycles intersect edge-cuts in even number of edges.

Thus all bridges of $H$ have colors from $\{a, b, c\}$. In particular, any bridge $e$ of $H$, which was adjacent to a loop of $G$, has a color $f(e)=d$, where $d\in \{a, b, c\}$. Color the loop $e'$ of $G$ that is adjacent to $e$ with $d'$, where $d'=a'$, $d'=b'$ or $d'=c'$, if $d=a$, $d=b$ or $d=c$, respectively. It is not hard to see that the described coloring meets the conditions (a) and (b) of the theorem.

The proof of the theorem is completed.
\end{proof}

\bigskip

Conjecture \ref{mainConjecture} states that all cubic graphs admit an $S$-coloring. On the other hand, in the previous theorem we have shown that all cubic pseudo-graphs have an $S_4$-coloring. One may wonder whether there is a statement analogous to these in the class of simple cubic graphs? More precisely, is there a connected simple cubic graph $H$ such that all simple cubic graphs admit an $H$-coloring? A natural candidate for $H$ is the graph $S_{16}$. Next we prove a theorem that justifies our choice of $S_{16}$. On an intuitive level it states that the only way of coloring the graph $S_{16}$ with some connected simple cubic graph $H$ is to take $H=S_{16}$. This result is analogues to the following theorem proved in \cite{PetersenRemark}. 

\begin{theorem} (V. V. Mkrtchyan, \cite{PetersenRemark})
Let ${G}$ be a connected cubic graph with  ${G \prec S}$, then ${G = S}$.
\end{theorem}

This is the precise formulation of our second result.

%% Result 2 %%
\begin{theorem}\label{S16thmUnique}
Let ${G}$ be a connected simple cubic graph with  ${G \prec S_{16}}$, then ${G = S_{16}}$.
\end{theorem}

\begin{proof}
As ${G \prec S_{16}}$ and ${S_{16}}$ has no a perfect matching,
then due to (c) of Lemma \ref{Lemma 1}, the graph $G$ also has no a perfect matching. 

Let ${f}$ be a ${G}$-coloring of ${S_{16}}$. If ${e \in E(G)}$, then we will say that ${e}$ is used (with respect to ${f}$),
if ${f^{-1}(e)\neq\emptyset}$. First of all, let us show that
if an edge ${e}$ of ${G}$ is used, then any edge adjacent to
${e}$ is also used.

So let ${e = (u, v)}$ be a used edge of ${G}$. For the sake of contradiction, assume that ${v}$ is incident to an edge ${z\in E(G)}$ that is not used. Assume that ${\partial_{G}(u) = \left\{a, b, e\right\} }$ (Figure \ref{theorem2f1}).

\begin{center}
\includegraphics [scale=0.6] {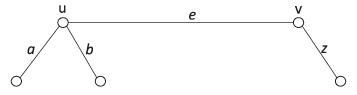}
\newcaption{The edge $e$ in the graph $G$}{theorem2f1}
\end{center}

We will use the labels of edges of $S_{16}$ from Figure \ref{SylvSimple}. The following cases are possible:
\medskip

\begin{case} 1: The edge ${e}$ colors an edge that is not a bridge in ${S_{16}}$. So it is used in end-blocks of ${S_{16}}$.
Due to symmetry of ${S_{16}}$, there will be following subcases:

\medskip

\begin{subcase} 1.1: ${f((v_1, v_2))=e}$. 

Since $z$ is not a used edge, we can assume that ${f((v_1, v_3))=a}$
and ${f((v_1,v_4))=b}$. Since adjacent edges receive different colors and $z$ is not a used edge, we have that ${f(
(v_2, v_3))=b}$ and  ${f((v_2, v_4))=a}$. This implies that
\begin{equation*}
f((v_3, v_5))=f((v_4, v_5))=e,
\end{equation*} which contradicts the fact that adjacent edges receive different colors.
\end{subcase}

\medskip

\begin{subcase} 1.2: ${f((v_1, v_3))=e}$.

Since $z$ is not a used edge, we can assume that ${f((v_1, v_2))=a}$
and ${f((v_1,v_4))=b}$. As adjacent edges receive different colors and $z$ is not a used edge, we have that ${f(
(v_2, v_3))=b}$ and  ${f((v_3, v_5))=a}$. But then ${f((v_2, v_4))=e}$, which implies that ${f((v_4, v_5))=a}$. This contradicts the fact that adjacent edges receive different colors.
\end{subcase}

\medskip

\begin{subcase} 1.3: ${f((v_3, v_5))=e}$. 

Since $z$ is not a used edge, we can assume that ${f((v_1, v_3))=a}$
and ${f((v_2,v_3))=b}$. Consider the edge $(v_1, v_2)$. Observe that its color can be either $e$, or there is an edge $h$ of $G$, such that $a$, $b$ and $h$ form a triangle in $G$ and $h$ is not incident to $u$. Since we have ruled out Case 1.1, we can assume that the color of $(v_1, v_2)$ is not $e$, hence there is the above-mentioned edge $h$. Let $x$ and $y$ be the edge of $G$ that are adjacent to $b$ and $h$, and $a$ and $h$, respectively, that are not incident to $u$. Observe that $f((v_1, v_4))=y$ and $f((v_2, v_4))=x$. On the other hand, it is not hard to see that since $z$ is not a used edge and ${f((v_3, v_5))=e}$, we have that $f((v_4, v_5))\in \{a, b\}$. This is a contradiction since there is no vertex $w$ of $G$ such that $\partial_G(w)=\{a, x, y\}$ or $\partial_G(w)=\{b, x, y\}$.

\end{subcase}
\end{case}

\medskip

\begin{case} 2: The edge ${e}$ colors an edge that is a bridge in ${S_{16}}$. 

The consideration of Case 1 implies that, without loss of generality, we can assume that the edge $e$ is not used in end-blocks of ${S_{16}}$. Assume that ${f((v,v_5))=e}$. Since $z$ is not a used edge, we can assume that ${f((v_3,v_5))=a}$ and ${f((v_4, v_5))=b}$.

Assume that $a=(u, u_a)$ and $b=(u, u_b)$. We claim that $u_a$ and $u_b$ are not joined with an edge in $G$. Assume the opposite. Let $h=(u_a, u_b)\in E(G)$. Let $x$ and $y$ be the edges of $G$ incident to $u_a$ and $u_b$, respectively, that are different from $a$ and $h$, and $b$ and $h$, respectively. Then, we can assume that 
\begin{equation*}
f((v_1, v_3))=x, f((v_2, v_3))=h,
\end{equation*} and
\begin{equation*}
f((v_2, v_4))=y, f((v_1, v_4))=h.
\end{equation*} This implies that
\begin{equation*}
a=f((v_1, v_2))=b.
\end{equation*} Hence $a$ and $b$ are parallel edges of $G$, which contradicts the simpleness of $G$. Hence $u_a$ and $u_b$ are not joined with an edge in $G$. Let $x$ and $y$ be the edges of $G$ incident to $u_a$, that are different from $a$. Similarly, let $z$ and $\alpha$ be the edges of $G$ incident to $u_b$, that are different from $b$. We can assume that
\begin{equation*}
f((v_1, v_3))=x, f((v_2, v_3))=y,
\end{equation*} and
\begin{equation*}
f((v_2, v_4))=z, f((v_1, v_4))=\alpha.
\end{equation*} This implies that $x$ and $\alpha$ are sharing a vertex $u_{x, \alpha}$ of $G$, $y$ and $z$ are sharing a vertex $u_{y, z}$ of $G$, and $u_{x, \alpha}$, $u_{y, z}$ are joined with an edge $g$ of $G$, such that $f((v_1, v_2))=g$. Observe that the edges $a$ and $b$ are lying on a cycle of $G$. Now, since $z$ is not a used edge, we have that the other two bridges of $G$ ($\neq (v,v_5)$) are colored with $a$ and $b$. This contradicts (e) of Lemma \ref{Lemma 1}, since $a$ and $b$ are not bridges of $G$.
\end{case}

\bigskip

The consideration of above two cases implies that any used edge of $G$ is adjacent to a used edge. Since $G$ is connected, we have that all edges of $G$ are used. Since $|E(S_{16})|=24$, we have that $|E(G)|\leq 24$, or $|V(G)|\leq 16$. Proposition \ref{property} implies that $G=S_{16}$. The proof of the theorem is completed.
\end{proof}

%\vspace{14mm}
%% Result 3 %%

\bigskip

In \cite{PetersenRemark}, the following result is obtained:

\begin{theorem} (V. V. Mkrtchyan, \cite{PetersenRemark})
Let ${G}$ be a connected bridgeless cubic graph with  ${G \prec P}$, then ${G = P}$.
\end{theorem}

Below we prove the analogue of this result for simple cubic graphs that may contain bridges. Our strategy of the proof is similar to that of given in \cite{PetersenRemark}.

\begin{theorem}\label{PetersenSimple}
If ${G}$ is a connected simple cubic graph with ${G \prec P}$, then ${G = P}$.
\end{theorem}

\begin{proof}
By (b) of Lemma \ref{Lemma 1}, $G$ is non-$3$-edge-colorable. Let $f$  be a $G$-coloring of $P$.
As in the proof of the previous theorem, we say that an edge ${e \in E(G)}$ is used (with respect to ${f}$) if ${f^{-1}(e)\neq\emptyset}$. First of all, let us show that if an edge of ${G}$ is used, then all edges adjacent to it are used.

Suppose that ${e = (u, v)}$ is a used edge, and for the sake of contradiction, assume that the 
edge ${z}$ incident to ${v}$ is not used. Assume that ${\partial_{G}(u) = \left\{a,b, e\right\} }$. We will make use of labels of vertices of $P$ given on Figure \ref{PetersenGraph}.

Since $e$ is a used edge, due to symmetry of $P$, we can assume that ${f((v_3,v_4)) = e}$, ${f((v_4,v_5)) = a}$ and ${f((u_4,v_4)) = b}$. Since $z$ is not a used edge, due to symmetry $P$, we can assume that ${f((u_3,v_3)) = b}$ , ${f((v_2,v_3)) = a}$.

Define 

\begin{equation*}
{a_1 = f((v_1,v_5))}, \text{ and } {a_2 = f((v_1,v_2))}.
\end{equation*} Observe that since $f$ is a $G$-coloring of $P$, we have that $a_1$ and $a_2$ are adjacent edges of $G$. Moreover, each of them is adjacent to $a$. 

Similarly, define the edges 
\begin{equation*}
{b_1 = f((u_1,u_4))}, \text{ and } {b_2 = f((u_1,u_3))}.
\end{equation*} Again, we have that $b_1$ and $b_2$ are adjacent edges of $G$. Moreover, each of them is adjacent to $b$.

We will consider three cases:

\bigskip

\begin{case} 1: The edges $a_1$, $a_2$ and $a$ do not form a triangle in $G$. 

Observe that in this case ${f((u_1,v_1)) = a}$. This implies that the edges $a$, $b_1$, $b_2$ must be
incident to the same vertex. However, this is possible only when $b_1$ and $b_2$ are two
parallel edges. This is a contradiction, since $G$ is a simple graph.
\end{case}

\bigskip

\begin{case} 2: The edges $b_1$, $b_2$ and $b$ do not form a triangle in $G$.

This case is similar to Case 1.
\end{case}

\bigskip

\begin{case} 3: The edges $a_1$, $a_2$ and $a$ form a triangle in $G$. Similarly, $b_1$, $b_2$
and $b$ form a triangle.

Let $a_3$ be the edge of $G$ that is adjacent to $a_1$, $a_2$
and is not adjacent to $a$. Similarly, let $b_3$ be the edge of $G$
that is adjacent to $b_1$, $b_2$ and is not adjacent to $b$. Note that the edges $a_3$ and $b_3$ exist, since $G$ is simple.

Observe that 

\begin{equation*}
{a_3 = f((u_1,v_1)) = b_3},
\end{equation*}hence ${a_3 =b_3}$. Depending on whether $a$ and $b$ belong to the same triangle
or not, we will consider the following sub-cases:

\medskip

\begin{subcase} 3.1: $a$ and $b$ belong to different triangles.

Observe that in this case the edge $e$ must belong to both of them, hence we have the situation depicted on Figure \ref{theorem3f1}. It is not hard to see that in this case ${a_3\neq b_3}$, which is a contradiction.
\end{subcase}

\vspace{13mm}

\begin{center}
  \begin{minipage}{0.4\textwidth}
  \hspace{10mm}
    \includegraphics[scale=0.6]{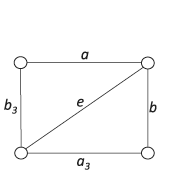}
    \newcaption{$a$ and $b$ belong to different triangles.}{theorem3f1}
  \end{minipage}
\hspace{5mm}
\begin{minipage}{0.4\textwidth}

\hspace{5mm}
    \includegraphics[scale=0.6]{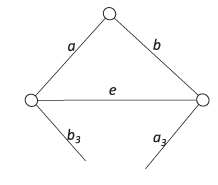}
    \newcaption{$a$ and $b$ belong to the same triangle.}{theorem3f2}
  \end{minipage}
  \end{center}
  
  \medskip
  
\begin{subcase} 3.2: Let $a$ and $b$ belong to the same triangle (See Figure \ref{theorem3f2}). 

In
this case $b_3$ should be adjacent to $a$, and $a_3$ should be adjacent to $b$. It is not hard to see that in this case ${a_3\neq b_3}$, which is a contradiction.
\end{subcase} 
\end{case}

\medskip

The consideration of above three cases implies that any used edge of $G$ is adjacent to a used edge. Since $G$ is connected, we have that all edges of $G$ are used. Since $|E(P)|=15$, we have that $|E(G)|\leq 15$, or $|V(G)|\leq 10$. Proposition \ref{Non3edgeProp} implies that $G=P$ or $G=S'$.

In order to complete the proof of the theorem, it suffices to show that $P$ does not admit an $S'$-coloring, such that all edges of $S'$ are used. We will use the labels of edges of $S'$ given on Figure \ref{theorem3f3}.

For the sake of the contradiction, assume that $P$ admits an $S'$-coloring $f$, such that all edges of $S'$ are used. Due to symmetry of $P$, we can assume that $f((v_3, v_4))=a$. Consider the connected components $A$ and $B$ of $S'-a$. Since $P$ is $2$-edge-connected, there is a vertex $w$ of $P$ such that $w$ is incident to at least one edge that has color from $A$, and at least one edge that has color from $B$. Observe that this vertex violates the definition of $S'$-coloring. This is a contradiction, hence $P$ does not admit an $S'$-coloring as well. The proof of the theorem is completed.
\end{proof}

\medskip

The theorem proved above implies that 

\begin{corollary} $P$ does not admit an $S_{16}$-coloring. 
\end{corollary} This corollary and Theorem \ref{S16thmUnique} suggest that a statement analogous to Sylvester coloring conjecture is impossible in the class of simple cubic graphs.

\bigskip

Our next result states that any cubic graph admits a coloring with edges of $S$, such that $80\%$ of vertices meet the constraints of Sylvester coloring conjecture. %In order to prove this statement, we need one lemma.

\begin{theorem}\label{SylvesterColResult} Let $G$ be a cubic graph. Then, there is a mapping $f:E(G)\rightarrow S$, such that
\begin{equation*}
|V(f)|\geq \frac{4}{5}\cdot |V(G)|,
\end{equation*}and for any $v \in V(G)\backslash V(f)$ there are two edges $e, e' \in \partial_G(v)$, such that $f(e)=f(e')$.
\end{theorem}

\begin{proof} We prove the theorem by induction on the number of vertices. If $|V(G)| = 2$, we can take an arbitrary vertex $w$ of $S$, and color the three edges of $G$ with edges incident to $w$. It is trivial to see that this coloring satisfies the condition of the theorem. Now, assume that the statement of the theorem holds for all cubic graphs with $|V(G)|<n$, and consider an arbitrary cubic graph $G$ containing $n\geq 4$ vertices.

We will consider two cases.

\bigskip

\begin{case} 1: $G$ contains a contractible triangle $T$.

\medskip

Consider the cubic graph $H=G/T$, and let $v_T$ be the vertex of $H$ obtained by contracting $T$. Since $H$ contains $n-2$ vertices, we have that there is a mapping $g:E(H)\rightarrow S$, such that
\begin{equation*}
|V(g)|\geq \frac{4}{5}\cdot |V(H)|,
\end{equation*}and for any $v \in V(H)\backslash V(g)$ there are two edges $e, e' \in \partial_H(v)$, such that $g(e)=g(e')$. We will consider $2$ subcases.

\medskip

\begin{subcase} 1.1: $v_T\in V(g)$.

\medskip

There is a vertex $s\in V(S)$, such that $g(\partial_H(v))=\partial_S(s)$. Let $\partial_S(s)=\{\alpha, \beta, \gamma \}$. Consider a mapping $f:E(G)\rightarrow S$, obtained from $g$ as follows: color the edges of $T$ with a color from $\{\alpha, \beta, \gamma \}$, such that its end-vertices are not incident to an edge with that color. Observe that
\begin{equation*}
|V(f)|=|V(g)|+2, \text{ and } |V(G)|=|V(H)|+2,
\end{equation*}hence
\begin{equation*}
\frac{|V(f)|}{|V(G)|}=\frac{|V(g)|+2}{|V(H)|+2}\geq \frac{|V(g)|}{|V(H)|}\geq \frac{4}{5},
\end{equation*}or
\begin{equation*}
|V(f)|\geq \frac{4}{5}\cdot |V(G)|.
\end{equation*}
\end{subcase}

\begin{subcase} 1.2: $v_T\notin V(g)$.
\medskip

There are two edges $e, e' \in \partial_H(v_T)$, such that $g(e)=g(e')$. Let $x=g(e)$, and let $y$ and $z$ be two edges of $S$ that are incident to the same end-vertex of $x$ in $S$.

Consider a mapping $f:E(G)\rightarrow S$, obtained from $g$ as follows: color the edges of $T$ that are opposite to the edges with color $x$ by $y$, and color the remaining third edge of $T$ with $z$. Observe that
\begin{equation*}
|V(f)|=|V(g)|+2, \text{ and } |V(G)|=|V(H)|+2,
\end{equation*}hence
\begin{equation*}
\frac{|V(f)|}{|V(G)|}=\frac{|V(g)|+2}{|V(H)|+2}\geq \frac{|V(g)|}{|V(H)|}\geq \frac{4}{5},
\end{equation*}or
\begin{equation*}
|V(f)|\geq \frac{4}{5}\cdot |V(G)|.
\end{equation*} Moreover, for each vertex $w \notin V(f)$, there are two edges $h, h' \in \partial_G(w)$, such that $f(h)=f(h')$.
\end{subcase}

\medskip

\end{case}

\medskip

\begin{case} 2: All triangles of $G$ are not contractible.

\medskip

Let $\mathcal{T}$ be the set of all triangles of $G$. Observe that $\mathcal{T}$ can be empty. Consider a graph $G'$ obtained from $G$ by removing all vertices of $G$ that lie on a triangle of $\mathcal{T}$. Observe that $G'$ is a triangle-free graph of maximum degree at most $3$. 

We will use the labels of edges of $S$ given in Figure \ref{SylvGraph}. Consider a coloring of edges of $G'$ with colors $a, b$ and $c$, such that the number of uncolored edges is smallest. Let $e$ be an uncolored edge. Color $e$ with a color $d$ from $\{a, b, c\}$, such that there is only one edge adjacent to $e$, such that it has also color $d$. Observe that all edges of $G'$ are colored. 

Now, we are going to extend this coloring to that of $G$. Choose a triangle $T$ from $\mathcal{T}$. As $T$ is not contractible, we have that the subgraph of $G$ induced by the vertices of $T$ form an end-block $B$ of $G$. Moreover, $B$ is isomorphic to end-blocks of $S$. Let $v$ the root of $B$. Choose a color $d\in \{a, b, c\}$ such that $d$ is missing on the vertex $v$ in the coloring of $G'$. Color the bridge joining a vertex of $T$ to $v$ by $d$, and color the edges of $B$ by corresponding edges of the end-block of $S$, which contains a vertex incident to $d$. Let $f$ be the resulting coloring.

Observe that all edges of $G$ are colored in $f$. Moreover, vertices of $V(G)\backslash V(f)$ lie in $G'$. Since $G'$ is triangle-free, we have that the cycles corresponding uncolored edges are of length at least $5$. Since they are vertex-disjoint (Lemma \ref{VertexDisjoint}), we have that their number is at most $\frac{|V(G')|}{5}$. It is not hard to see that each uncolored edge $e$ is incident to a vertex $v$ such that $v\in V(G)\backslash V(f)$. Moreover, $|V(G)\backslash V(f)|$ coincides with the number of uncolored edges, which implies that 
\begin{equation*}
|V(G)\backslash V(f)|\leq \frac{|V(G')|}{5} \leq \frac{|V(G)|}{5},
\end{equation*}or
\begin{equation*}
|V(f)|\geq \frac{4}{5}\cdot |V(G)|.
\end{equation*} Finally, for each vertex $w \notin V(f)$, there are two edges $h, h' \in \partial_G(w)$, such that $f(h)=f(h')$.

\end{case}

The proof of the theorem is completed.
\end{proof}

\begin{corollary}Let $G$ be a cubic graph. Then, there is a mapping $f:E(G)\rightarrow S$, such that 
\begin{equation*}
|V(f)|\geq \frac{4}{5}\cdot |V(G)|.
\end{equation*}
\end{corollary}

In the end of the paper, we verify Conjecture \ref{mainConjecture} in the class of claw-free cubic graphs. Our main ingredients are the characterization of claw-free cubic graphs (Proposition \ref{prop:ClawfreeCharac}) and Theorem \ref{thm:PseudoCubicS4} about $S_{4}$-colorability of arbitrary cubic pseudo-graphs. 

\begin{center}
    \includegraphics [scale=0.8] {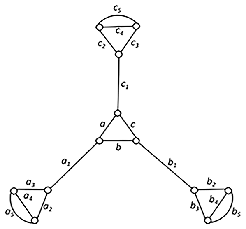}
    \newcaption{The graph $S_{12}$}{fig:Sylvester12}
\end{center}

Let $S_{12}$ be the cubic graph from Figure \ref{fig:Sylvester12}. We prove somewhat stronger statement.

\begin{theorem}\label{thm:ClawFreeS12} Let $G$ be a claw-free cubic graph. Then $S_{12}\prec G$.
\end{theorem}

\begin{proof} Our proof is by induction on $|V(G)|$. Clearly, the statement of the theorem is true when $|V(G)|=2$ ($K_{2}^{3}$ is $3$-edge-colorable). Assume that it remains true for all claw-free cubic graphs with $|V(G)|<n$, and consider a claw-free cubic graphs with $|V(G)|=n$. Without loss of generality, we can assume that $G$ is connected. 

We can apply Proposition \ref{prop:ClawfreeCharac}. If $G$ meets the conditions (1) or (2) of the proposition, then $G$ is $3$-edge-colorable, hence this case is similar to the base of induction. Thus, we can assume that $G$ meets the condition (3) of the Proposition \ref{prop:ClawfreeCharac}. 

Let us show that we can assume that $G$ all strings of diamonds and $2$-cycles of $G$ are trivial. On the opposite assumption, consider a non-trivial string $J$ of diamonds and $2$-cycles of $G$. Let $a$ and $b$ be the head and tail of $J$, respectively. Moreover, let $c$ and $d$ be the neighbors of $a$ and $b$, respectively, that lie outside $J$. If $c\neq d$, then consider a cubic graph $G'$ defined as follows:
\[G'=(G-V(J))+(c,d).\]
Observe that $G'$ is a claw-free cubic graph with $|V(G')|<n$, hence by induction hypothesis, it admits an $S_{12}$-coloring $g$. Let $g((c,d))=\alpha$, where $\alpha$ is an edge of $S_{12}$. Moreover, let $\beta$ and $\gamma$ be $2$ edges of $S_{12}$ leaving the same end-vertex of $\alpha$ in $S_{12}$. 

Color the edges $(a,c)$ and $(b,d)$ with $\alpha$. Since rings of diamonds and $2$-cycles are $3$-edge-colorable, we can color the edges of $J$ with $\alpha$, $\beta$ and $\gamma$, so that each vertex of $J$ is incident to edges with colors $\alpha$, $\beta$ and $\gamma$. It can be easily checked that this new coloring is an $S_{12}$-coloring of $G$.

If $c=d$, then since $G$ is claw-free, we have that $a$ and $b$ are joined by $2$ parallel edges, hence $J$ is a trivial string contradicting our assumption. 

Thus all strings of diamonds or $2$-cycles of $G$ are trivial. This and (3) of Proposition \ref{prop:ClawfreeCharac} imply that there is a cubic pseudo-graph $H$, such that $G$ can be obtained from $H$ by replacing any vertex of $H$ with a triangle. By Theorem \ref{thm:PseudoCubicS4}, $H$ admits an $S_4$-coloring such that its loops are colored by loops of $S_4$ (see (a) of Theorem \ref{thm:PseudoCubicS4}). Now, observe that $S_{12}$ can be obtained from $S_4$ by replacing any vertex of $S_4$ by a triangle. 

Extend the $S_4$-coloring of $H$ to an $S_{12}$-coloring of $G$ by coloring the edges of new triangles of $G$ by the edges of the corresponding new triangles of $S_{12}$. One can easily see that there is always a way of doing this, which results to an $S_{12}$-coloring of $G$. 

The proof of the theorem is completed.
\end{proof}
\medskip

Taking into account that $S\prec S_{12}$, and $\prec$ is transitive, we have the following corollary of Theorem \ref{thm:ClawFreeS12}:

\begin{corollary} Let $G$ be a claw-free cubic graph. Then $S \prec G$.
\end{corollary}

%% References


%\bibliographystyle{elsarticle-num}
% \bibliographystyle{elsarticle-harv}
% \bibliographystyle{elsarticle-num-names}
% \bibliographystyle{model1a-num-names}
% \bibliographystyle{model1b-num-names}
% \bibliographystyle{model1c-num-names}
% \bibliographystyle{model2-names}
% \bibliographystyle{model3a-num-names}
% \bibliographystyle{model3-num-names}
% \bibliographystyle{model4-names}
% \bibliographystyle{model5-names}
% \bibliographystyle{model6-num-names}

\end{document}